\documentclass[12pt,a4paper]{amsart}

\usepackage[utf8]{inputenc}
\usepackage{enumerate,bbm,titletoc,a4,mathtools,amssymb,xspace,amsthm}
\usepackage[usenames,dvipsnames]{xcolor}
\usepackage[colorlinks,linkcolor=BrickRed,citecolor=CadetBlue,urlcolor=black,hypertexnames=true]{hyperref}
\usepackage[all]{xy}

\DeclareMathOperator{\tr}{tr}
\DeclareMathOperator{\Tr}{Tr}
\DeclareMathOperator{\opm}{M}
\DeclareMathOperator{\ops}{S}
\DeclareMathOperator{\opd}{D}
\newcommand{\ve}{\varepsilon}

\newcommand{\N}{\mathbb{N}}

\newcommand{\R}{\mathbb{R}}

\newcommand{\cK}{\mathcal{K}}

\newcommand{\cO}{\mathcal{O}}

\newcommand{\addots}{\reflectbox{$\ddots$}}
\newcommand{\han}{\mathfrak{H}}
\newcommand{\sT}{{\bf T}}
\newcommand{\fT}{\mathbb{T}}

\newcommand{\ti}{{\rm t}}

\newcommand*{\sym}[1]{\ops_{#1}(\R)}
\newcommand*{\diag}[1]{\opd_{#1}(\R)}
\newcommand{\ort}{\operatorname{O}_n(\R)}

\makeatletter
\newtheorem*{rep@thm}{\rep@title}
\newcommand{\newreptheorem}[2]{%
\newenvironment{rep#1}[1]{%
 \def\rep@title{#2 \ref{##1}}%
 \begin{rep@thm}}%
 {\end{rep@thm}}}
\makeatother

\newtheorem{thm}{Theorem}[section]
\newreptheorem{thm}{Theorem}
\newtheorem{lem}[thm]{Lemma}
\newtheorem{cor}[thm]{Corollary}
\newreptheorem{cor}{Corollary}
\newtheorem{prop}[thm]{Proposition}

\theoremstyle{definition}

\newtheorem{exa}[thm]{Example}
\theoremstyle{remark}
\newtheorem{rem}[thm]{Remark}

\numberwithin{equation}{section}

\title[Positive univariate trace polynomials]{
	Positive univariate trace polynomials}

\author[I. Klep]{Igor Klep${}^1$}
\address{Igor Klep, Faculty of Mathematics and Physics, University of Ljubljana}
\email{igor.klep@fmf.uni-lj.si}
\thanks{${}^1$Supported by the Slovenian Research Agency grants J1-2453 and P1-0222.}

\author[J. Pascoe]{James Eldred Pascoe${}^2$}
\address{James Eldred Pascoe, Department of Mathematics, University of Florida}
\email{pascoej@ufl.edu}
\thanks{${}^2$Supported by the NSF grants DMS-1953963 and DMS-1606260.}

\author[J. Vol\v{c}i\v{c}]{Jurij Vol\v{c}i\v{c}${}^3$}
\address{Jurij Vol\v{c}i\v{c}, Department of Mathematics, Texas A\&M University}
\email{volcic@math.tamu.edu}
\thanks{${}^3$Supported by the NSF grant DMS-1954709.}

\date{\today}

\begin{document}
	
\begin{abstract}
A univariate trace polynomial is a polynomial in a variable $x$ and formal trace symbols $\Tr(x^j)$. Such an expression can be naturally evaluated on matrices, where the trace symbols are evaluated as normalized traces. This paper addresses global and constrained positivity of univariate trace polynomials on symmetric matrices of all finite sizes. A tracial analog of Artin's solution to Hilbert's 17th problem is given: a positive semidefinite univariate trace polynomial is a quotient of sums of products of squares and traces of squares of trace polynomials.
\end{abstract}

\maketitle


\section{Introduction}

Univariate trace polynomials are real polynomials in $x$ and $\Tr(x^j)$ for $j\in\N$. We view $x$ as a matrix variable of unspecified size, and evaluate a (univariate) trace polynomial $f(x,\Tr(x),\Tr(x^2),\dots)$ at any $n\times n$ matrix $X$ as $f(X,\frac1n\tr(X),\frac1n\tr(X^2),\dots)$. That is, the trace symbol $\Tr$ evaluates as the normalized trace of a matrix. Trace polynomials as matricial functions originated in invariant theory \cite{Pro}, and more recently emerged in free probability \cite{GS} and quantum information theory \cite{PNA,FN}. In this paper we characterize trace polynomials that have positive semidefinite values at symmetric matrices of all sizes.

Trace polynomials form a commutative polynomial ring (in countably many variables), and several sum-of-squares positivity certificates (Positivstellens\"atze) for multivariate polynomials on semialgebraic sets are provided by real algebraic geometry \cite{BCR,Las,Mar,Lau,BPT}. However, this theory does not appear to directly apply to our setup. First, matrix evaluations of trace polynomials are just a special class of homomorphisms from trace polynomials. Second, the dimension-free context addresses positivity on symmetric matrices of all sizes, hence on a countable disjoint union of real affine spaces;
there is no bound (with respect to the degree of a trace polynomial) on the size of matrices for which positivity needs to be verified (Remark \ref{r:bd}).

Therefore a different approach is required. 
To demonstrate it, consider the inequality
\begin{equation}\label{e:ex}
\Tr(X^4)(\Tr(X^2)-\Tr(X)^2)+2\Tr(X^3)\Tr(X^2)\Tr(X)-\Tr(X^3)^2-\Tr(X^2)^3\ge0
\end{equation}
which holds for all symmetric matrices $X$. One way to certify \eqref{e:ex} is by noticing that 
$f=\Tr(x^4)(\Tr(x^2)-\Tr(x)^2)+2\Tr(x^3)\Tr(x^2)\Tr(x)-\Tr(x^3)^2-\Tr(x^2)^3$ is the determinant of the Hankel matrix
$$\begin{pmatrix}
1 & \Tr(x) & \Tr(x^2) \\
\Tr(x) & \Tr(x^2) & \Tr(x^3) \\
\Tr(x^2) & \Tr(x^3) & \Tr(x^4)
\end{pmatrix}$$
which is positive semidefinite for every matrix evaluation (since it is obtained by applying the normalized partial trace to a positive semidefinite matrix). Another certificate of \eqref{e:ex}, in the spirit of sum-of-squares representations in real algebraic geometry, is
\begin{align*}
&\ \Tr\Big(\big(x-\Tr(x)\big)^2\Big) \cdot f \\
= &\ \Tr\Big(\big(
(\Tr(x)^2-\Tr(x^2)) x^2+(\Tr(x^3)-\Tr(x^2)\Tr(x)) x+\Tr(x^2)^2-\Tr(x^3)\Tr(x)
\big)^2\Big),
\end{align*}
where we view $\Tr$ as an idempotent linear endomorphism of trace polynomials in a natural way.  
Thus $f$ is a quotient of traces of squares. The main result of this paper shows that these characterizations apply to all positive trace polynomials.

\begin{repcor}{c:artin}
Let $f$ be a univariate trace polynomial. Then $f(X)$ is positive semidefinite for all symmetric matrices $X$ if and only if $f$ is a quotient of sums of products of squares and traces of squares of trace polynomials.
\end{repcor}

Corollary \ref{c:artin} is a special case of Theorem \ref{t:posss}, a tracial Positivstellensatz that characterizes positivity of trace polynomials subject to tracial constraints under certain mild regularity assumptions. The proof of Theorem \ref{t:posss} splits into two parts: every tracial inequality is a consequence of a certain tracial Hankel matrix being positive semidefinite (Proposition \ref{p:main}), and this positive semidefiniteness is in turn certified by traces of squares (Proposition \ref{p:sigma}).

Since we are addressing trace polynomials in only one matrix variable and the trace is invariant under conjugation, we could of course restrict evaluations to diagonal matrices and reach the same positivity conclusions.
From this viewpoint, Corollary \ref{c:artin} pertains to positive symmetric polynomials and mean inequalities in combinatorics and statistics \cite{Bul,Tim,CGS,MS,Sra}: a positive trace polynomial corresponds to a sample-size independent power mean (or moment) inequality.

Nevertheless there are benefits to working with general matrix evaluations of trace polynomials. Besides the algebraic structure, there is an intimate connection with the emerging area of free analysis \cite{KVV} (cf. Proposition \ref{p:ax}). 
Evaluations on arbitrary symmetric matrices put the positivity of {\it univariate} trace polynomials under the umbrella of {\it multivariate} trace polynomials and noncommutative tracial inequalities induced by them. If one considers only tuples of matrices of fixed size, Positivstellens\"atze on arbitrary tracial semialgebraic sets are known \cite{PS0,PS,Cim,KSV}. On the other hand, multivariate trace positivity on matrices of all finite sizes is not understood well. Namely, the failure of Connes' embedding conjecture \cite{JNVWY} implies that there is a noncommutative polynomial whose trace is positive on all matrix contractions, but negative on a tuple of operator contractions from a tracial von Neumann algebra \cite{KS}, which obstructs the existence of a clean trace-of-squares certificate for matrix positivity in general. There are however Positivstellens\"atze for positivity of multivariate trace polynomials on von Neumann algebras subject to archimedean constraints \cite{KMV}. In a different direction, tracial inequalities of analytic functions are heavily studied in relation to monotonicity, convexity and entropy in quantum statistical mechanics \cite{Car}. 
The results of this paper
fill the gap in understanding univariate trace polynomials 
by demonstrating that in one operator variable, global positivity on matrices of all sizes implies global operator positivity (in tracial von Neumann algebras).

\section{Univariate trace polynomials}

In this section we introduce terminology and notation that will be used throughout the paper. This includes the notion of a preordering from real algebraic geometry \cite{Mar}, and Proposition \ref{p:sigma} establishes a relation between two preorderings appearing in our positivity certificate.

Let $\sT=\R[\Tr(x^j)\colon j\in\N]$ be the polynomial ring generated by countably many independent symbols $\Tr(x^j)$. Its elements are called (univariate) {\bf pure trace polynomials}. By adjoining an additional variable $x$ to $\sT$ one obtains the ring of (univariate) {\bf trace polynomials} $\fT = \sT \otimes \R[x]$. Let $\Tr:\fT\to\sT$ denote the unital $\sT$-linear map given by $x^j\mapsto\Tr(x^j)$ for all $j\in\N$.

Let $\sym{n}$ denote the space of $n\times n$ real symmetric matrices. We consider matrix evaluations of trace polynomials, where the trace symbols are evaluated using the \emph{normalized} matrix trace on $\sym{n}$. For example, if
$$f=x^2-2\Tr(x)x+2\Tr(x)^2-\Tr(x^2) \qquad\text{and}\qquad X\in\sym{n}$$
then
$$f(X) = X^2-\frac{2\tr(X)}{n}X+\left(\frac{2\tr(X)^2}{n^2}-\frac{\tr(X^2)}{n}\right)I_n\in\sym{n}.$$
One can further elaborate on this viewpoint of a trace polynomial as a family of matricial polynomial functions that are equivariant under the orthogonal conjugation, respect ampliations and have a common degree bound. This perspective is inspired by free analysis \cite{KVV,KS1}, where free functions are equivariant and respect direct sums instead of just ampliations. An analog of the following proposition also holds for multivariate trace polynomials (which are not considered in this paper).

\begin{prop}\label{p:ax}
A sequence $(f_n)_n$ of polynomial maps $f_n:\sym{n}\to\sym{n}$ is given by a trace polynomial if and only if
\begin{enumerate}[(i)]
\item $f_n(OXO^\ti)=Of_n(X)O^\ti$ for every $n\in\N$, $X\in\sym{n}$ and $O\in\ort$;
\item $f_{kn}(I_k\otimes X)=I_k\otimes f_n(X)$ for all $k,n\in\N$ and $X\in\sym{n}$;
\item $\sup_n\deg f_n<\infty$.
\end{enumerate}
\end{prop}

\begin{proof}
On the polynomial ring $\fT$, consider the degree function $\deg$
given by $\deg x=1$ and $\deg \Tr(x^j)=j$ for $j\in\N$.

$(\Rightarrow)$ If $f\in\fT$ and $\deg f=d$, then the polynomial maps $f_n = f|_{\sym{n}}$ clearly satisfy (i) and (ii), and $\deg f_n\le d$ for all $n\in\N$.

$(\Leftarrow)$ Since $f_n$ is $\ort$-equivariant, there exists $\tilde{f}_n\in\fT$ with $\deg f_n=\deg \tilde{f}_n$ such that $f_n = \tilde{f}_n|_{\sym{n}}$, see e.g. \cite[Theorem 7.3]{Pro}. Let $\deg f_n\le d$ for all $n\in\N$. By (iii), for all $k\in\N$ and $X\in\sym{d+1}$ we have
$$I_k\otimes \tilde{f}_{k(d+1)}(X)=f_{k(d+1)}(I_k\otimes X)=I_k \otimes f_{d+1}(X)=I_k \otimes \tilde{f}_{d+1}(X),$$
so $\tilde{f}_{k(d+1)}=\tilde{f}_{d+1}$ for all $k$ since there are no tracial identities for $\sym{d+1}$ of degree $d$ \cite[Proposition 8.3]{Pro}. A further application of (iii) then gives
$$I_{d+1}\otimes f_k(X)
= f_{(d+1)k}(I_{d+1}\otimes X)
= I_{d+1}\otimes \tilde{f}_{(d+1)k}(X)
= I_{d+1}\otimes \tilde{f}_{d+1}(X)$$
for all $k\in\N$ and $X\in\sym{k}$. Therefore the trace polynomial $f=\tilde{f}_{d+1}$ yields the sequence $(f_n)_n$.
\end{proof}

A direct consequence of Proposition \ref{p:ax} is the description of pure trace polynomials as sequences of polynomial functions.

\begin{cor}
A sequence $(p_n)_n$ of polynomial functions $p_n:\sym{n}\to\R$ is given by a pure trace polynomial 
if and only if
\begin{enumerate}[(i)]
	\item $p_n$ is $\ort$-invariant for every $n\in\N$;
	\item $p_{kn}(I_k\otimes X)=p_n(X)$ for all $k,n\in\N$ and $X\in\sym{n}$;
	\item $\sup_n\deg p_n<\infty$.
\end{enumerate}
\end{cor}

To a set of constraints $S\subset\fT$ we associate its {\bf positivity set}
$$\cK_S = 
\bigcup_{n\in\N}\left\{
X\in\sym{n}\colon
s(X)\succeq0 \text{ for all } s\in S
\right\}.
$$

\subsection{Preorderings and Hankel matrices}

Given a commutative ring $R$ and $S\subset R$, the {\it preordering} generated by $S$ is the smallest set $P\subseteq R$ containing $S$ that is closed under addition and multiplication, and $r^2\in P$ for every $r\in R$ \cite[Section 2.1]{Mar}.

For $j=1,\dots,m$ let $\sigma_j(X)\in R$ be such that
\begin{equation}\label{e:char}
t^m-\sigma_1(X)t^{m-1}+\cdots+(-1)^{m-1}\sigma_{m-1}(X)t+(-1)^m\sigma_m(X) \in R[t]
\end{equation}
is the characteristic polynomial of $X\in \opm_m(R)$. In other words, $\sigma_j(X)$ is the trace of the $j$th exterior power of $X$, and in particular $\sigma_1(X)=\tr(X)$. Note that when $R$ is a real closed field, $X$ is positive semidefinite if and only if $\sigma_1(X),\dots,\sigma_m(X)\ge0$ by the Descartes rule of signs.

Let $d\in\N$. We say that a $(d+1)\times(d+1)$ (symmetric) matrix $H$ is a {\it Hankel matrix} if $H_{11}=1$ and $H_{i_1j_1}=H_{i_2j_2}$ for $i_1+j_1=i_2+j_2$. The Hankel matrix over $\sT$ whose $(i,j)$-entry equals $\Tr(x^{i+j-2})$ for $1\le i,j\le d+1$ is denoted $\han_d$:
$$\han_d = 
\begin{pmatrix}
1 & \Tr(x) & \Tr(x^2) & \cdots & \Tr(x^d) \\
\Tr(x) & \addots & \addots & \addots & \Tr(x^{d+1}) \\
\Tr(x^2) & \addots & \addots & \addots & \vdots \\
\vdots & \addots & \addots & \addots & \vdots \\
\Tr(x^d) & \Tr(x^{d+1}) & \cdots & \cdots & \Tr(x^{2d})
\end{pmatrix}.
$$
For each $d\in\N$ we consider finitely generated polynomial rings
$$\sT_d=\R[\Tr(x^j)\colon j\le d]\subset\sT\qquad\text{and}\qquad \fT_d = \sT_d\otimes \R[x]\subset\fT.$$
Let $\Pi_d$ be the preordering in $\sT_{2d}$ generated by $\{\sigma_j(\han_d)\colon j=1,\dots,d+1 \}$, 
and let $\Omega$ be the preordering in $\sT$ generated by $\{\Tr(f^2)\colon f\in\fT \}$. As demonstrated in the next section, both preorderings can be used to certify positivity. The upside of $\Omega$ is that it does so in terms of squares, which are the most basic building blocks of positivity. However, $\Omega$ is not a finitely generated preordering (which would not change even if only polynomials of bounded degree were considered). On the other hand, $\Pi_d$ is a finitely generated preordering, and thus more in line with classical results in real algebraic geometry.

\begin{lem}\label{l:tr}
Let $A\in\opm_m(\fT)$ and $C_1,C_2\in\opm_m(\sT)$. If $B\in\opm_m(\sT)$ is the matrix obtained by applying $\Tr$ to $A$ entry-wise, then
$$\tr(C_1BC_2)=\sum_{k=1}^m \Tr\big((C_1AC_2)_{kk}\big).$$
\end{lem}

\begin{proof}
Straightforward.
\end{proof}

The following proposition shows that, inside the field of fractions of pure trace polynomials, $\Pi_d$ is generated by $\Omega$.

\begin{prop}\label{p:sigma}
$\sigma_j(\han_d)$ is a quotient of elements from $\Omega$ for every $j,d\in\N$ with $j\le d+1$.
\end{prop}

\begin{proof}
Let $\Xi$ be the generic $(d+1)\times (d+1)$ symmetric matrix; that is, entries of $\Xi$ are commuting indeterminates, related only by $\Xi$ being symmetric. Let $A$ be the real polynomial algebra generated by the entries of $\Xi$, and let $R$ be its real subalgebra generated by $\{\sigma_i(\Xi)\colon 1\le i\le d+1 \}$. Let $P$ be the preordering in $R$ generated by
$$\{\tr(h(\Xi)^2),\,\tr(h(\Xi)^2\Xi)\colon h\in\fT \}.$$
By \cite[Lemma 4.1 and Theorem 4.13]{KSV} there exist $p_1,p_2\in P$ and $k\in\N$ such that 
\begin{equation}\label{e:pos}
p_1\sigma_j(\Xi)=\sigma_j(\Xi)^{2k}+p_2.
\end{equation}
Let $w^\ti = (1\, x\, \cdots \, x^d)$; then $\han_d$ is obtained by applying $\Tr$ to $ww^\ti \in\opm_{d+1}(\R[x])$ entry-wise. If $h\in\fT$, then $h(\han_d)\in \opm_{d+1}(\sT)$ and $h(\han_d)w \in \fT^{d+1}$; hence by Lemma \ref{l:tr},
\begin{equation}\label{e:hanksquare}
\tr\left(h(\han_d)^2\,\han_d\right)=\tr\left(h(\han_d)\,\han_d\,h(\han_d)\right) =
\sum_{k=1}^{d+1} \Tr\big(\left(h(\han_d)w\right)_k\left(h(\han_d)w\right)_k^\ti\big) \in\Omega
\end{equation}
If $q_i$ is obtained from $p_i$ by replacing $\Xi$ with $\han_d$, then $q_1,q_2\in \Omega$ by \eqref{e:hanksquare}. Further, $q_1\neq0$ by \eqref{e:pos} since the right-hand side is strictly positive when evaluated at a positive definite $(d+1)\times (d+1)$ Hankel matrix. Therefore $\omega_1=\sigma_j(\han_d)^{2k}+q_2$ and $\omega_2=q_1$ satisfy $\sigma_j(\han_d) = \frac{\omega_1}{\omega_2}$ by \eqref{e:pos}.
\end{proof}

For Proposition \ref{p:sigma} it is crucial that $\Omega$ is generated not only by traces of squares of polynomials in $\R[x]$, but traces of squares of trace polynomials; see Example \ref{ex:tr} below. Furthermore, quotients in Proposition \ref{p:sigma} are necessary, as demonstrated by the following example.

\begin{exa}\label{e:frac}
Let
\begin{align*}
f=\sigma_2(\han_2) 
&= \Tr(x^4)\Tr(x^2)-\Tr(x^3)^2+\Tr(x^4)-\Tr(x^2)^2+\Tr(x^2)-\Tr(x)^2 \\
&= \frac{\Tr\big((\Tr(x^2)x^2-\Tr(x^3)x)^2\big)}{\Tr(x^2)}
+\Tr\big((x^2-\Tr(x^2))^2\big)
+\Tr\big((x-\Tr(x))^2\big).
\end{align*}
On the polynomial ring $\fT$, consider the degree function $\deg$
given by $\deg x=1$ and $\deg \Tr(x^j)=j$ as in the proof of Proposition \ref{p:ax}.
Then $\deg f=6$. 
Furthermore, one observes that if $\omega\in\Omega$ with $\deg \omega=6$ contains the monomial $\Tr(x^3)^2$ with a negative sign, then it must contain the monomial $\Tr(x^6)$ with a positive sign. Since this fails for $f$, we conclude that $f\notin\Omega$.
\end{exa}

\subsection{Univariate trace polynomials as multivariate polynomials}

We will frequently view trace polynomials in $\fT_d$ as polynomials in $d+1$ variables $$t_0=x,\ t_1=\Tr(x),\ \dots,\ t_d=\Tr(x^d).$$
Thus they can evaluated at points in the real affine space $\R^{1+d}$. To avoid confusion with matrix evaluations, we write $f[\gamma]\in\R$ for such an evaluation of $f\in \fT_d$ at $\gamma\in\R^{1+d}$. Furthermore, if $f\in\sT_d$, then we also evaluate it at $\gamma'\in \R^d$ as $f[\gamma']$, as the inclusion of rings $\sT_d\subset\fT_d$ corresponds to the projection $\R^{1+d}=\R\times\R^d\to \R^d$.

\begin{exa}\label{ex:tr}
If $f=\sigma_2(\han_1)=\det(\han_1)$, then
$$f=\Tr(x^2)-\Tr(x)^2 = \Tr\big((x-\Tr(x))^2\big) \in\Omega.$$
On the other hand, we claim that $f$ cannot be written as a quotient of elements from 
the preordering in $\sT$ generated by $\{\Tr(p^2)\colon p\in\R[x] \}$; that is, trace polynomials are essential for sum-of-squares representations.

Suppose this is not true. Note that $f$ is a quadratic in the generators of $\sT$, while $\Tr(p^2)$ for $p\in\R[x]$ are linear in generators of $\sT$. Let $d\in\N$ be such that $f$ is a quotient of elements from the preordering in $\sT_{2d}$ generated by $\{\Tr(p_1^2),\dots,\Tr(p_\ell^2)\}$ for some $p_1,\dots,p_\ell\in\R[x]$ of degree at most $d$. As above, we view the generators $t_j=\Tr(x^j)$ as coordinates of the real affine space $\R^{2d}$ corresponding to the polynomial ring $\sT_{2d}$. Observe that then $\det(\han_1[\gamma])\ge0$ implies $\han_1[\gamma]\succeq0$ since $(\han_1)_{11}=1$. Therefore
\begin{equation}\label{e:polyhedron}
\begin{split}
&\ \{\gamma\in\R^{2d}\colon \han_1[\gamma]\succeq 0 \} \\
= &\ \{\gamma\in\R^{2d}\colon f[\gamma]\ge0 \} \\
\supseteq &\ \{\gamma\in\R^{2d}\colon\Tr(p_1^2)[\gamma]\ge0,\dots,\Tr(p_\ell^2)[\gamma]\ge0 \}=:W \\
\supseteq &\ \{\gamma\in\R^{2d}\colon\Tr(p^2)[\gamma]\ge0 \text{ for all } p\in\R[x],\ \deg p\le d\} \\
= &\ \{\gamma\in\R^{2d}\colon \han_d[\gamma]\succeq 0 \}.
\end{split}
\end{equation}
Note that $W$ is a polyhedron. Since the projections of $\{\gamma\colon \han_1[\gamma]\succeq 0 \}$ and $\{\gamma\colon \han_d[\gamma]\succeq 0 \}$ onto the $(t_1,t_2)$-plane coincide, the inclusions \eqref{e:polyhedron} imply that the projections of $W$ and $\{\gamma\colon f[\gamma]\ge0\}$ onto the $(t_1,t_2)$-plane also coincide. But then the projection of $\{\gamma\colon f[\gamma]\ge0\}$ onto this plane is a polyhedron, which is impossible since $f$ is quadratic.
\end{exa}

Finally, to a finite set $S\subset\sT_{2d}$ we assign the basic closed semialgebraic set
$$L_S = 
\{\gamma\in\R^{2d}\colon s[\gamma]\ge0 \text{ for all } s\in S \text{ and }
\han_d[\gamma]\succeq0
\}.$$

\section{Positivstellensatz for univariate trace polynomials}

In this section we prove our main result (Theorem \ref{t:posss}) which describes trace polynomials that are positive on tracial semialgebraic sets subject to certain mild regularity assumptions. As a corollary we characterize globally positive trace polynomials (Corollary \ref{c:artin}). Examples justifying the assumptions in Theorem \ref{t:posss} are also given.

\subsection{Main result}

At the core of our Positivstellensatz is the observation that the points in $L_S$, which originate from tracial evaluations on matrices in $\cK_S$, are dense in the interior of $L_S$. More precisely, we require the following statement, which relies on a solution of the truncated moment problem \cite{CF1}. For a classical application of moment problems for deriving rational sum-of-squares certificates of positivity, see \cite{PV}.

\begin{prop}\label{p:main}
Let $S\subset\sT_{2d}$ be a finite set and $g\in\fT_{2d}\setminus\{0\}$. For every interior point $\beta\in \R\times L_S\subset\R^{1+2d}$ and $\ve>0$ there exists $X\in\cK_S$ with an eigenpair $(\lambda,v)$ such that
\begin{enumerate}[(i)]
	\item $|\beta_j-\Tr(X^j)|<\ve$ for $j=1,\dots,2d$;
	\item $|\beta_0-\lambda|<\ve$;
	\item $v^\ti g(X)v\neq0$.
\end{enumerate}
\end{prop}

\begin{proof}
Without loss of generality we can assume that $0\notin S$. Then the interior of $L_S$ is contained in the closure of
$$\{\gamma\in\R^{2d}\colon s[\gamma]>0 \text{ for all } s\in S \text{ and } \han_d[\gamma]\succ0  \}.$$
After an arbitrarily small perturbation of $\beta$ we can therefore assume that
$$s[\beta]>0 \text{ for all } s\in S,\ \han_d[\beta]\succ0 \text{ and } g[\beta]\neq0.$$
For positive definite Hankel matrices, the truncated Hamburger moment problem is solvable \cite[Theorems I.1 and I.3]{AK}; cf. \cite[Theorem 3.9]{CF}. Thus there exists a $(d+1)$-atomic measure $\mu$ on $\R$ representing $\han_d[\beta]$, i.e., there are $\alpha_1,\dots,\alpha_{d+1}\in\R$ and $\lambda_1,\dots,\lambda_{d+1}\in\R_{>0}$ with $\sum_i\lambda_i=1$ such that $\mu=\sum_i\lambda_i\delta_{\alpha_i}$, where $\delta$ stands for the Dirac measure, satisfies
$$\beta_j=\int t^j\, {\rm d}\mu,\qquad j=1,\dots,2d.$$

Next we approximate $\lambda_i$ by rationals with a denominator that is much larger than $\beta_0$.
That is, for every $\tilde\ve>0$ there exist $n,m_1,\dots,m_{d+1} \in\N$ such that
\begin{enumerate}[(1)]
	\item $|n\lambda_i-m_i|< n \tilde\ve$ for $i=1,\dots,d+1$;
	\item $\sum_i m_i = n-1$;
	\item $|\beta_0|< n \tilde\ve$.
\end{enumerate}
Denote
$$\tilde\mu=\frac{1}{n}\delta_{\beta_0}+\sum_i\frac{m_i}{n}\delta_{\alpha_i}$$
and set
$$\tilde\beta_0=\beta_0 \qquad \text{and}\qquad
\tilde\beta_j=\int t^j\, {\rm d}\tilde\mu \qquad \text{for } j=1,\dots,2d.$$
Since $S$ is finite and we are approximating finitely many values, we can choose $\tilde\ve$ small enough compared to $\ve$ such that
\begin{enumerate}[(1')]
	\item $|\beta_j-\tilde\beta_j|<\ve$ for $j=0,\dots,2d$;
	\item $(\tilde\beta_1,\dots,\tilde\beta_{2d})\in L_S$;
	\item $g[\tilde\beta]\neq0$.
\end{enumerate}
The $n\times n$ diagonal matrix
$$\tilde X= \tilde\beta_0 I_1\oplus \bigoplus_i\alpha_i I_{m_i}$$
satisfies $\frac{1}{n}\tr(\tilde{X}^j) = \tilde\beta_j$ for $1\le j\le 2d$. Consequently $s(X)\ge 0$ for all $s\in S$. Furthermore, $(\tilde\beta_0,e_1)$ is an eigenpair for $X$ and $e_1^\ti g(X)e_1=g[\tilde\beta]\neq0$. Hence (i)--(iii) are satisfied.
\end{proof}

A subset of a topological space is \emph{regular closed} if it equals the closure of its interior.

\begin{thm}\label{t:posss}
Given a finite set $S\subset\sT_{2d}$ 
let $P$ be the preordering in $\fT_{2d}$ generated by $S\cup \Pi_d$,
and let $Q$ be the preordering in $\fT$ generated by $S\cup\Omega$.
Suppose that $L_S$ is regular closed in $\R^{2d}$.
Then the following are equivalent for $f\in\fT_{2d}$:
\begin{enumerate}[(i)]
	\item $f(X)\succeq0$ for all $X\in \cK_S$;
	\item there are $p_1,p_2\in P$ and $k\in\N$ such that
	\begin{equation}\label{e:ks}
	p_1f=f^{2k}+p_2;
	\end{equation}
	\item $f[\beta]\ge0$ for all $\beta\in \R\times L_S$;
	\item there is $q\in Q\setminus\{0\}$ such that $qf \in Q$.
\end{enumerate}
\end{thm}

\begin{proof}
(iii)$\Rightarrow$(ii) Since the positive semidefiniteness of $\han_d$ is certified by $\sigma_j(\han_d)\ge0$ for $j=1,\dots,d+1$ and $P$ is a finitely generated preordering in a finitely generated polynomial ring, this implication is an instance of the Krivine--Stengle Positivstellensatz \cite[Theorem 2.2.1]{Mar}.

(i)$\Rightarrow$(iii) Follows by Proposition \ref{p:main} (the roles of $g$ and the eigenpair are irrelevant in this instance).

(ii)$\Rightarrow$(iv) Suppose \eqref{e:ks} holds. Without loss of generality assume that $f\neq0$; then $p_1\neq0$ by \eqref{e:ks} since $L_S$ has nonempty interior. By Proposition \ref{p:sigma}, every element of $P$ can be written as a quotient of an element in $Q$ and an element from $\Omega$. So $f$ is a quotient of elements from $Q$ by \eqref{e:ks}.

(iv)$\Rightarrow$(i) Suppose $qf=r$ for $q,r\in Q$ with $q\neq0$, and let $X\in\cK_S$. Let $(\lambda,v)$ be an eigenpair of $X$. Since $L_S$ is regular closed, by Proposition \ref{p:main} there exists a sequence of matrices $X_k \subset \cK_S$ with eigenpairs $(\lambda_k,v_k)$ such that
\[
\begin{split}
\lim_{k\to\infty}\lambda_k &=\lambda, \\
\lim_{k\to\infty}\Tr(X_k^j) &=\Tr(X^j), \qquad j=1,\dots,2d, \\
v_k^\ti q(X_k)v_k &\neq0 \qquad k\in\N.
\end{split}
\]
Since  $q(X_k)f(X_k)=r(X_k)$ and $q(X_k),r(X_k)\succeq0$, we have $v_k^\ti f(X_k)v_k\ge0$ for all $k\in\N$. On the other hand,
\begin{align*}
v^\ti f(X) v
&= f[\lambda,\Tr(X),\dots,\Tr(X^{2d})] \\
&= \lim_{k\to\infty} f[\lambda_k,\Tr(X_k),\dots,\Tr(X_k^{2d})] \\
&= \lim_{k\to\infty} v_k^\ti f(X_k)v_k
\end{align*}
and therefore $v^\ti f(X) v\ge0$. As the eigenpair $(\lambda,v)$ of $X$ was arbitrary, $f(X)$ is positive semidefinite.
\end{proof}

\begin{rem}\label{r:vna}
The regular closed assumption in Theorem \ref{t:posss} is indispensable. Note that sum-of-squares certificates imply not only positivity of matrix evaluations, but also positivity of evaluations in tracial von Neumann algebras. However, there exist pure tracial constraints that are infeasible for matrices of arbitrary finite size, but feasible for infinite-dimensional tracial von Neumann algebras. For example, let
\begin{equation}\label{e:bad}
s_1=\Tr((x-x^2)^2),\qquad s_2 = \sqrt{2}\Tr(x)-1.
\end{equation}
No $X\in\sym{n}$ satisfies $s_1(X)=s_2(X)=0$ (as if $s_1(X)=0$, then $X$ is a projection, so $\Tr(X)\in\mathbb{Q}$). On the other hand, viewing $L^\infty([0,1])$ as an abelian von Neumann algebra with the trace given by the Lebesgue integration, the characteristic function $\chi$ for $[0,\frac{1}{\sqrt{2}}]$ satisfies $s_1(\chi)=s_2(\chi)=0$. Therefore the conclusion (i)$\Rightarrow$(ii) of Theorem \ref{t:posss} fails for $S=\{\pm s_1, \pm s_2\}$ and $f=-1$.
A Positivstellensatz for (multivariate) trace polynomials positive on operators in tracial von Neumann algebras satisfying archimedean constraints is given in \cite[Corollary 4.8]{KMV}.
\end{rem}

\begin{rem}
The denominators $p_1$ and $q$ in parts (ii) and (iv) of Theorem \ref{t:posss} are necessary in general. For (iv), this is demonstrated by Example \ref{e:frac}. 
For (ii), this holds even if the larger preordering $P'$ in $\fT_{2d}$ finitely generated by $S$ and all the principal minors of $\han_d$ is considered. Note that $\fT_{2d}$ is the polynomial ring in $2d+1\ge3$ indeterminates, and the semialgebraic set corresponding to $P'$ equals $\R\times L_S$ and has nonempty interior by the assumption. Then by \cite[Proposition 2.6.2]{Mar} there exists $f\in\fT_{2d}$ that is non-negative on $\R\times L_S$ but $f\notin P'$.
\end{rem}

\begin{rem}\label{r:fix}
When only matrices of a given \emph{fixed} size $n$ are considered, then multivariate trace polynomials that are positive on all $n\times n$ matrices are characterized via traces of squares in the seminal work of Procesi and Schacher \cite{PS0}. The extension of this characterization to positivity on invariant semialgebraic subsets of $n\times n$ matrices is given in \cite{KSV}. However, one cannot deduce our size-independent positivity certificate from such positivity certificates for each fixed size $n$.
The obstacle is the non-existence of a bound on $n$ such that positivity of a trace polynomial $f$ on $n\times n$ matrices would imply positivity of $f$ on all matrices; see Remark \ref{r:bd} below. Thus one cannot simply deduce Theorem \ref{t:posss} from the existing dimension-dependent Positivstellens\"atze \cite{PS0,KSV}.
\end{rem}

\begin{rem}\label{r:bd}
In Theorem \ref{t:posss}, there is no bound on the size $n$ of $X\in \cK_S$ in terms of $d$ for which positivity of $f(X)$ need to be tested to ensure positivity on whole $\cK_S$.
We demonstrate this for $d=2$ and $S=\emptyset$. Let
$$f=\Tr((x-x^2)^2)+(\sqrt{2}\Tr(x)-1)^2 \in \sT_4.$$
We claim that for every $n\in\N$ there exists $\ve>0$ such that $(f-\ve)|_{\sym{n}}\ge0$ but $(f-\ve)|_{\sym{N}}\neq0$ for some $N>n$. Indeed, fix $n$ and let $m>1$ be such that $m(m-1)\ge \sqrt{n}$. Since $f$ is nonzero on $\sym{n}$, $f(0)=1$ and $f(X)\ge 1$ for $X\in\sym{n}$ with $\|X\|\ge m$, $f$ admits a global minimum $\ve>0$ on $\sym{n}$. Therefore $f-\ve\ge0$ on $\sym{n}$. On the other hand, there are $r,N\in\N$ such that $(\sqrt{2}\frac{r}{N}-1)^2<\ve$, so a projection $X\in\sym{N}$ of rank $r$ satisfies $f(X)-\ve<0$. 

Nevertheless, the proof of Theorem \ref{t:posss} (or rather Proposition \ref{p:main}) shows that it suffices to consider only matrices $X$ with at most $d+2$ distinct eigenvalues. In terms of positivity of trace polynomials on finite-dimensional real von Neumann algebras $\R^n$: if one tests positivity with respect to any possible tracial state on $\R^n$ as opposed to only the averaging state, it suffices to check $n=d+2$.
\end{rem}

\subsection{Convex tracial constraints}

The most compelling part of Theorem \ref{t:posss} is the equivalence (i)$\Leftrightarrow$(iv), which refers only to positivity of matrix evaluations and (traces of) squares, but not to coefficients of the characteristic polynomial of the Hankel matrix $\han_d$. Nevertheless, Theorem \ref{t:posss} is only valid under the assumption that the semialgebraic set $L_S$ is regular closed, so Hankel matrices still tacitly lurk around. The following lemma gives a class of constraints where verifying that $L_S$ is regular closed does not require detailed information about the geometry of $L_S$. 

\begin{lem}\label{l:cvx}
Let $S\subset\sT_{2d}$ be finite.
\begin{enumerate}[(a)]
	\item If $\cK_S\cap\sym{n}$ has nonempty interior in $\sym{n}$ for some $n\in\N$, then $L_S$ has nonempty interior in $\R^{2d}$.
	\item If $S$ is given by a linear matrix inequality of the symbols $\Tr(x^j)$ and $\cK_S\cap\sym{n}$ has nonempty interior for some $n\in\N$, then $L_S$ is regular closed.
\end{enumerate}
\end{lem}

\begin{proof}
(a) Let $S=\{s_1,\dots,s_\ell\}$. Since $\cK_S\cap\sym{n}$ has nonempty interior, there is $X\in\sym{n}$ such that $s_i(X)\succ0$ for $i=1,\dots,\ell$. Consequently $s_i(I_k\otimes X)=I_k\otimes s_i(X)\succ0$ for $i=1,\dots,\ell$ for every $k\in\N$. Thus we can without loss of generality assume that $n\ge 2d$. Furthermore, if $\diag{n}$ denotes the space of $n\times n$ diagonal matrices, then $\cK_S\cap\diag{n}$ has nonempty interior in $\diag{n}$. Therefore there exists $D_0$ in the interior of $\cK_S\cap\diag{n}$ with pairwise distinct eigenvalues. Let $\cO\subset\cK_S\cap\diag{n}$ be an open neighborhood of $D_0$ such that $P\cO P^{-1}\cap\cO=\emptyset$ for every permutation $n\times n$ matrix $P$. Consider the polynomial map
$$\phi:\diag{n}\to \R^{2d},\qquad D\mapsto (\Tr(D),\dots,\Tr(D^{2d})).$$
Clearly $\phi(\cK_S\cap\diag{n})\subseteq L_S$. Also, the restriction of $\phi$ to $\cO$ is an open map since it is a composition of
$$\cO\to \R^n,\qquad D\mapsto (\Tr(D),\dots,\Tr(D^n))$$
which is open by the Invariance of domain theorem, and the projection $\R^n\to\R^{2d}$. Hence $\phi(D_0)$ is an interior point of $L_S$.

(b) Since $L_S$ is the solution set of a linear matrix inequality (as $\han_d\succeq0$ is already a linear matrix inequality), it is convex. Furthermore, $L_S$ has nonempty interior by (a), so it is regular closed. 
\end{proof}

As a consequence we obtain a tracial version of Artin's solution of Hilbert's 17th problem.

\begin{cor}\label{c:artin}
Let $f\in\fT$. Then $f(X)\succeq0$ for all $X\in\bigcup_n\sym{n}$ if and only if $f$ is a quotient of sums of products of squares and traces of squares of trace polynomials.
\end{cor}

\begin{proof}
Apply Lemma \ref{l:cvx} and Theorem \ref{t:posss} with $S=\emptyset$.
\end{proof}

\begin{rem}\label{r:sdp}
While sum-of-squares certificates with denominators as in Theorem \ref{t:posss}(iv) are seldom of immediate use in polynomial optimization,
our methods nevertheless yield a practical algorithm for optimizing the normalized trace of a univariate polynomial over matrices of arbitrary sizes. Let $M$ be a symmetric matrix over $\sT_{2d}$ whose entries are linear in trace symbols, and let $S$ be a set of constraints describing $M\succeq0$. Then Proposition \ref{p:main} implies that, for $p\in\R[x]$ of degree at most $2d$, one can compute
$$\inf\big\{\Tr(p(X)) \colon X\in\cK_S \big\}$$
by solving the semidefinite program in $2d$ variables
$$\min\ \Tr(p)[\gamma] \quad\text{subject to}\quad \han_d[\gamma]\succeq0\ \& \ M[\gamma]\succeq0,$$
which can be done efficiently using interior point methods \cite{BPT}.
\end{rem}

\subsection{Impure tracial constraints}

In Theorem \ref{t:posss}, only pure tracial constraints are allowed; that is, $S\subset\sT$. A natural extension to $S\subset \fT$ could consider the preordering in $\fT$ generated by
$$S\cup\{\tr(g^2),\, \tr(g^2s)\colon g\in\fT,\, s\in S \}$$
or rather its variation with Hankel matrices.
For $s\in\fT$ let $\han_d^s$ be the $(d+1)\times(d+1)$ Hankel matrix whose $(i,j)$-entry is $\Tr(x^{i+j-2}s)$ (cf. localizing Hankel matrix \cite[Section 4.1.3]{Lau}). Observe that
$\han_d^s\succeq0$ corresponds to $\Tr(q^2s)\ge0$ for all $q\in \R[x]$ of degree at most $d$.
Thus one could consider, for each $d\in\N$, the preordering in $\fT$ generated by
\begin{equation}\label{e:d2}
S\cup\{\sigma_j(\han_d),\, \sigma_j(\han_d^s)\colon s\in S,\, 1\le j\le d+1 \}.
\end{equation}
This approach indeed leads to an algebraic certificate of positivity on $\cK_S$ as in \eqref{e:ks} if $S=S'\cup\{x\}$ or $S=S'\cup\{x-a,b-x\}$ for $S'\subset\sT$ and $a<b$. The reasoning is analogous to the proof of Theorem \ref{t:posss}: in Proposition \ref{p:main}, the truncated Hamburger problem is now replaced by the truncated Stieltjes and Hausdorff problems, respectively, which are also solvable under the positive definiteness assumption \cite{CF}. We leave the details to the reader. However, for a general $S\subset \fT$, the positivity on $\cK_S$ cannot be certified with preorderings generated by \eqref{e:d2}, as demonstrated by the following example.

\begin{exa}
Let $P\subseteq \sT$ be the preordering generated by
$$\{\sigma_j(\han_d),\, \sigma_j(\han_d^{x^3})\colon j,d\in\N,\, j\le d+1 \}.$$	
For every $X\in\sym{n}$, $X^3\succeq0$ implies $\Tr(X)\ge0$. On the other hand, we claim that there are no $p_1,p_2\in P$ and $k\in\N$ such that
$$p_1\Tr(x)=\Tr(x)^{2k}+p_2.$$
To see this it suffices to prove the following: for every $d\in\N$ there exists $\alpha\in\R^{2d+4}$ such that
\begin{equation}\label{e:x3}
\han_{d+2}[\alpha]\succeq 0 \quad\text{and}\quad\han_d^{x^3}[\alpha]\succeq0 \quad\text{and}\quad \alpha_1<0.
\end{equation}
Fix $d\in\N$, and let $\mu$ be a $(d+2)$-atomic measure on $[0,\infty)$. Let $\beta_1=-1$ and $\beta_i = \int t^{i-2}\,d\mu$ for $i=2,\dots, 2d+4$. The Stieltjes moment problem \cite[Example 3.1.8]{Mar} implies
\begin{equation}\label{e:B}
B:=\left(\beta_{i+j} \right)_{i,j=1}^{d+2} \succeq 0
\quad\text{and}\quad
\left(\beta_{i+j+1}\right)_{i,j=1}^{d+1} \succeq 0.
\end{equation}
Furthermore, $B\succ0$ since $\mu$ has enough atoms. Let $b^\ti=(\beta_1\, \cdots \, \beta_{d+2}).$ Since $B$ is positive definite, there exists $\ve>0$ such that $B-\ve bb^\ti\succeq0$. Then $\alpha = \ve\beta\in\R^{2d+4}$ satisfies \eqref{e:x3}.
\end{exa}


\end{document}